\documentclass[12 pt]{article}
\usepackage{amsfonts}
\usepackage{amsthm}
\usepackage{amsmath}
\usepackage{enumerate}
\usepackage{graphicx} 
\usepackage{setspace}

\usepackage{authblk}
\usepackage[titletoc,title]{appendix}

\def\inte{\int\limits}

\def\ss{\kern.03cm}

\newtheorem{thm}{Theorem}
\theoremstyle{definition}
\theoremstyle{definition}\newtheorem{lemma}{Lemma}

\author[$\dagger$]{Eugenio P. 
Balanzario\thanks{ebg@matmor.unam.mx}}
\affil[$\dagger$]{Centro de 
Ciencias Matem\'aticas,
Universidad Nacional Aut\'onoma 
de M\'exico,
Apartado Postal 61-3 (Xangari),
Morelia  Michoac\'an, M\'exico}
\author[$\dagger$]{Daniel Eduardo
C\'ardenas Romero}

\begin{document}

\title{
An explicit formula for
the zeros of the Riemann zeta function
}

\maketitle

\begin{abstract}
We present an explicit formula
for a weighted sum over the
zeros of the Riemann zeta function. 
This weighted sum is evaluated in
terms of a sum over the prime numbers,
weighted with help of the Hermite polynomials.
From the explicit formula presented in
this note, it follows that
prime numbers determine the distribution of
the zeros of the Riemann zeta function.
In fact, it is possible to compute the
zeros of the zeta function without actually using
this function. 
\end{abstract}

\section*{Introduction}

Let $\Lambda(n)$ be the von Mangoldt arithmetical
function and $w(t)$ be
a real valued function defined for $t>0$.
Let us consider the function 
$m(y)$ defined by
\begin{eqnarray*}
m(y)=\sum_{n=1}^\infty
\Lambda(n)\kern.03cm w(n\ss y)
\end{eqnarray*}
for $y>0$ and
where $w(t)$ serves a weight function
for $\Lambda(n)$.
Assuming $w(t)$ is well behaved,
the Mellin transform of $m(y)$ is
given by
\begin{eqnarray*}
\widehat{m}(s)=
\inte_0^\infty y^{s-1}m(y)\>dy
=-\frac{\zeta'(s)}{\zeta(s)}
\widehat{w}(s)
\end{eqnarray*}
for $s=\sigma+it$ such that
$\sigma>1$ and where $\zeta(s)$ is
the Riemann zeta function.
From Perron inversion formula,
\begin{eqnarray}\label{mellin}
m(y)=\lim_{q\to\infty}
\frac{-1}{2\pi i}
\inte_{2-iq}^{2+iq}
\widehat{w}(s)\ss 
\frac{\zeta'(s)}{\zeta(s)}\ss
y^{-s}\>ds.
\end{eqnarray}
By shifting the vertical line
of integration to the left, we
obtain formally, for $x\geq0$,
that
\begin{eqnarray}\label{worth}
\sum_{n=1}^\infty
\Lambda(n)\kern.03cm 
w\big(\frac{n}{x}\big)
=
x\ss\widehat{w}(1)
-\sum_\rho
x^\rho\ss
\widehat{w}(\rho)+E
\end{eqnarray}
where $\rho=\beta+i\gamma$ stands
for a non trivial zero of $\zeta(s)$ and
$E$ stands for contributions
to the integral (\ref{mellin}) other
than those included in the first two terms
on the right hand side of 
equation (\ref{worth}).

At this point, it is desirable to choose
a weight function $w(t)$ in such a way
that equation (\ref{worth}) is
worth studying. In a previous 
note \cite{Balanzario},
we have chosen $w(t)$ to be a probability
density function depending on two
parameters. With these two parameters we were
able to place the bulk of the probability mass
given by $w(t)$
at any preassigned place in the positive
real line, with the result that we could 
isolate a finite number of terms
on the left hand side of equation (\ref{worth}).
In the extreme case that only one term is
isolated on the left hand side of 
equation (\ref{worth}) we obtain as
a result a smooth version
of Landau's explicit formula \cite{Landau},
\[
\Lambda(x)=-\frac{2\pi}{T}
\sum_{0<\lambda\leq T}
x^\rho + O\Big(
\frac{\log T}{T}\Big).
\]

In this note we choose $w(t)$ to be such
that its Mellin transform $\widehat{w}(s)$ is 
equal to a probability density function 
over the so called critical line,
which is the vertical line 
defined by $\sigma=1/2$.
Also, $\widehat{w}(s)$ will be such that
$\widehat{w}(1/2-it)=\widehat{w}(1/2+it)$
for each $t\in\mathbb{R}$.
With $\widehat{w}(s)$ depending on two free
parameters, we will be able to locate
half of the unit mass given 
by $\widehat{w}(1/2+it)$ around a small
neighborhood of a point
on the line $s=1/2+it$ with $t>0$. The second
half of the mass will be located symmetrically
along the line segment $s=1/2-it$ with
$t>0$. Thus, by letting this 
half unit mass be sufficiently 
concentrated around a suitable point,
a single zero, together with its
conjugate, will be isolated
from the sum over $\rho$ on
the right hand side of equation (\ref{worth}).

In this note we consider a
weight function $w(t)$ such that its
Mellin transform is given by
\begin{eqnarray}\label{defWM}
\widehat{w}(s)=
c_1\kern.04cm(-1)^k
\Big(s-\frac{1}{2}\Big)^{2k}
\exp\Big\{\frac{k}{\xi^2}
\Big(s-\frac{1}{2}
\Big)^2\Big\}
\kern.4cm\hbox{where}\kern.4cm
c_1=\frac{
(k/\xi^2)^{k+\frac{1}{2}}}{
\Gamma(k+\frac{1}{2})}
\end{eqnarray}
for every $s\in\mathbb{C}$.
We note
that $\widehat{w}(s)$ depends
on $\xi>0$ and $k\in\mathbb{N}$
as two free parameters.

Through all this note we assume the 
validity of the Riemann hypothesis ($\mathrm{H}_1$)
as well as the hypothesis affirming 
that all zeros of the Riemann 
zeta function are simple ($\mathrm{H}_2$).

Now we state the main result of this note.

\begin{thm}
For $\xi>0$ and $k\in\mathbb{N}$, let
\begin{eqnarray}\label{parametros}
S(\xi,k)=\frac{1}{\xi}
\sqrt{\frac{k}{2\pi}}
\sum_{n=1}^\infty
\exp\bigg\{
-2\frac{k}{\xi^2}(\gamma_n-\xi)^2
\bigg\}
\end{eqnarray}
where $\gamma_n$ is the
imaginary part of the $n$-th zero of
the Riemann zeta function in the upper
half plane. Assuming hypotheses $\mathrm{H}_1$
and $\mathrm{H}_2$, then we have
\begin{eqnarray}\label{formula}
S(\xi,k)=
\frac{1}{4\pi}\log\frac{\xi}{2\pi}-
c_2\sum_{n=1}^\infty
\frac{\Lambda(n)}{\sqrt{n}}
\exp\bigg\{
-\frac{\xi^2}{4k}\log^2n\bigg\}
H_{2k}\Big(\xi\frac{\log n}{2
\sqrt{k}}\Big)
+E
\end{eqnarray}
with $E\ll(\sqrt{k}/\xi^2+
\exp\{-\sqrt{k}/\xi\}
)\log\xi$
as $\xi\to\infty$ and where
\[
c_2=
\frac{1}{2\sqrt{\pi}}
\frac{(-1/4)^k}{
\Gamma(k+1/2)}
\kern.8cm\hbox{and}\kern.8cm
H_k(x)=(-1)^ke^{x^2}\frac{d^k}
{dx^k}e^{-x^2}
\]
is the $k$-nth the Hermite polynomial.
\end{thm} 

Let $\beta>1$ be a constant real
number. In the next section we will 
find it convenient to
put $k=\beta(\xi\log(\xi\log\xi))^2$, and in this
case the error term $E$ in the above theorem
is given by (as $\xi\to\infty$)
\[
E\ll\frac{\log^2\xi}{\xi}.
\]

\section*{Determining
the zeros of $\zeta(s)$}

In this section,
we address the question of how
many terms in the sum on the right hand side of
formula (\ref{formula}) are necessary in order
to obtain a good agreement between both
sides of the equation. We do this
by using the the following rough approximation
to the Hermite polynomials 
\[
H_n(x)\approx e^{\frac{1}{2}x^2}
\frac{2^n}{\sqrt{\pi}}
\Gamma\Big(\frac{n+1}{2}\Big)
\cos\Big(x\sqrt{2n}-\frac{n\pi}{2}\Big)
\]
(\cite{Abramowitz}, 
pages 508-510, formulas 13.6.38 and 13.5.16).
From this approximation to the
Hermite polynomials,
we obtain the following 
approximate expression for
the weight function
\begin{eqnarray}\label{aprox}
\widetilde{w}(x):=
\frac{1}{2\pi}\frac{1}{\sqrt{x}}
\exp\bigg\{-\frac{1}{8\alpha}
\log^2x\bigg\}
\cos(\xi
\log(x)-k\ss\pi)
\end{eqnarray}
where here and
in the remainder of this note 
we write $\alpha=k/\xi^2$.
From equation (\ref{aprox})
we have
\[
|\widetilde{w}(x)|\leq
\exp\Big\{
-\frac{1}{8\alpha}
\log^2x\Big\}.
\]
Given a small number $\epsilon>0$,
let $\tau=\tau(\epsilon)$ be the number of terms
in the sum on the right hand side of
formula (\ref{formula}) needed in order to
estimate $S(\xi,k)$ within an error not
greater than $\epsilon$. For the
estimation of $\tau$, notice that 
\begin{eqnarray*}
\sum_{n>\tau}
\Lambda(n)\ss|\widetilde{w}(n)|
&\leq&
\inte_\tau^\infty
\log(x)\ss
\exp\bigg\{-\frac{1}{8\alpha}
\log^2x\bigg\}\>dx\\
&=&
4\alpha\exp\Big\{-
\Big(\frac{\log\tau}{2\sqrt{2\alpha}}
-\sqrt{2\alpha}\Big)^2+2\alpha\Big\}.
\end{eqnarray*}
Equating this expression to $\epsilon$,
clearing for $\tau$ and writing $k/\xi^2_0$
in place of $\alpha$, we see 
that for the sum $\sum_{n>\tau}
\Lambda(n)\ss|\widetilde{w}(n)|$ to be less than
or equal to $\epsilon$, it suffices if
\begin{eqnarray}\label{tau}
\tau=
\exp\bigg\{
4\bigg(
\frac{k}{\xi^2_0}+
\sqrt{\frac{k}{\xi_0^2}}
\sqrt{\frac{k}{\xi_0^2}+
\frac{1}{2}\log
\frac{4k}{\epsilon\ss\xi_0}}\bigg)
\bigg\}.
\end{eqnarray}

\begin{figure}
\begin{center}
\includegraphics[width=13.5cm]{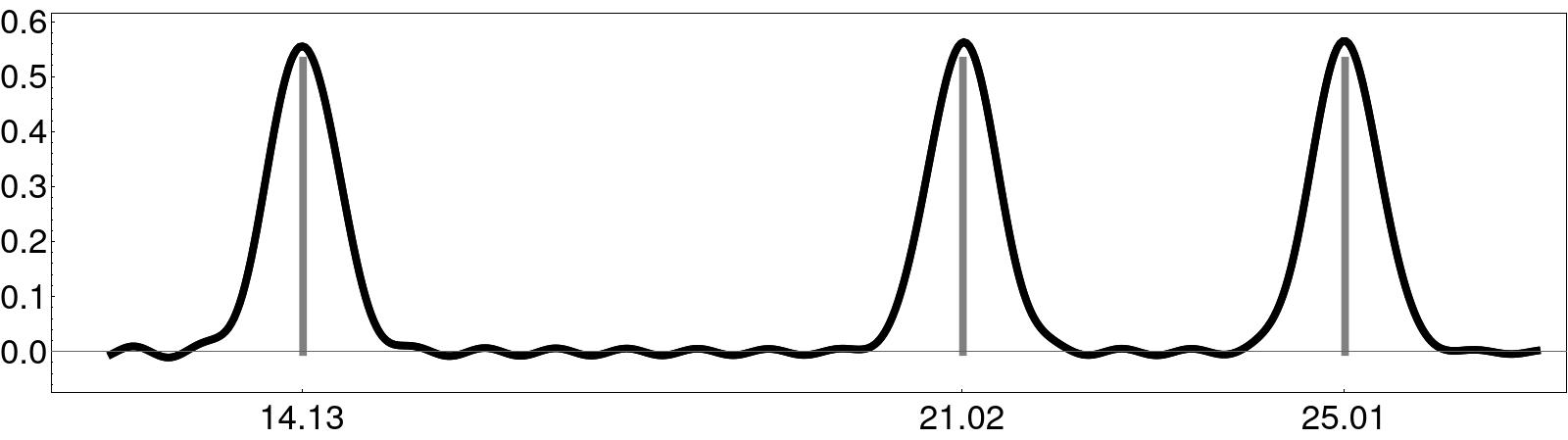}
\end{center}
\caption{The graph of $S(\xi,2\xi^2)$
where $S(\xi,k)$ is as in equation
(\ref{parametros}).}
\label{fig1}
\end{figure}

From equation (\ref{parametros}) it
follows that $S(\xi,k)$ places
approximately half a unit mass
in every sufficiently small neighborhood 
of $\xi_0$ if and only if  
$\xi_0=\gamma_n$ for some
$n\in\mathbb{N}$.
This fact is illustrated in 
figure~\ref{fig1} where the graph was
computed by using the right hand
side of formula (\ref{formula})
but with $\widetilde{w}(x)$, as given in 
equation (\ref{aprox}), in place
of $w(x)$. In this figure, 
the imaginary part of the first three zeros
of the Riemann zeta function are marked
with a vertical line. We appreciate
the accumulation of half a unit mass
around the imaginary parts of
zeros of the Riemann zeta function.
This accumulation of mass allows us
to decide whether or not a given
number $\xi_0$ is equal to the imaginary
part of one of the zeros of the
Riemann zeta function. 

Around each point $\xi_0=\gamma_n$ with
$n\in\mathbb{N}$,
the local variance of $S(\xi,k)$ is equal to
$\sigma^2=\gamma_n^2/8k$. On the other hand,
for the imaginary part of the $n$-th zero of
the Riemann zeta function it is known 
that the following asymptotic relation
takes place
(see \cite{Ivic}, page 20) 
\[
\gamma_n\sim\frac{2\pi\ss n}
{\log n}
\]
as $n\to\infty$.
Hence, in order to place half a unit mass
in a small neighborhood of $\gamma_n$,
the standard deviation of $S(\xi,k)$
around such a neighborhood of $\xi_0=\gamma_n$
must be of an order of magnitude of
\[
\frac{\xi_0}{\sqrt{k}}=
\frac{1}{\sqrt{\beta}\log n}
\]
where $\beta>1$ is moderately large but fixed.
The larger $\beta$ is, the greater the half mass
will be concentrated around $\xi_0$. 
Putting $k=\beta(\xi_0\log n)^2$ in
equation (\ref{tau}) we obtain
\[
\tau=
\exp\Big\{4\Big(\beta\log^2n+
\sqrt{\beta}\log n\sqrt{
\beta\log^2 n+\log2\sqrt{\beta/\epsilon}
\log n}\Big)\Big\}.
\]
If $\epsilon$ is fixed, then we have that
$\tau\sim\exp\{4\beta\log^2n\}$ as $n\to\infty$.
With this expression we have that
$\exp\{\beta'\log^2(\xi_0\log \xi_0)\}$ is 
an estimate for 
the number of primes necessary in order to
determine whether a number $1/2+i\xi_0$ is
a zero of the Riemann zeta function
for each $\xi_0\in\mathbb{R}$. 

In this section we have seen that,
assuming the Riemann hypothesis and the
simplicity of the zeros of $\zeta(s)$,
knowledge of the distribution of
prime numbers and its integral powers is
sufficient for the determination of the
distribution of the nontrivial 
zeros of the Riemann zeta function.
Thus, it is possible to compute the
zeros of $\zeta(s)$ without actually using
this function. However, 
a large number of primes are necessary
for the determination of the 
zeros of $\zeta(s)$.

\section*{Proof of theorem 1}

\begin{lemma}
Let $\widehat{w}(s)$ be is as in (\ref{defWM}),
$c_2$ as in theorem 1 and $\alpha=k/\xi^2$.
Then
\begin{eqnarray}\label{pedro}
w(x)=
\frac{c_2}{\sqrt{x}}\kern.03cm
\exp\bigg\{
-\frac{1}{4\alpha}\log^2 x\bigg\}
H_{2k}\Big(\frac{\log x}{2
\sqrt{\alpha}}\Big).
\end{eqnarray}
\end{lemma}

\begin{proof}
Let us consider the integral
\[
I=
\inte_{\sigma_0-i\infty}^{\sigma_0+i\infty}
\Big(s-\frac{1}{2}\Big)^{2k}
e^{\alpha(s-\frac{1}{2})^2}e^{-sx}\>ds
\kern.7cm\hbox{where}\kern.7cm
\sigma_0=\frac{1}{2}
\Big(1+\frac{x}{\alpha}\Big).
\]
We must show that $I=2\pi i\ss c_1^{-1}(-1)^k
\ss w(e^x)$ with $w(x)$ as given
in equation (\ref{pedro}). Now,
\begin{equation*}
\begin{split}
I&=\exp\bigg\{\frac{\alpha}{4}
-\alpha\Big(\frac{1}{2}+
\frac{x}{2\alpha}\Big)^2\bigg\}
\inte_{\sigma_0-i\infty}^{\sigma_0+i\infty}
\Big(s-\frac{1}{2}\Big)^{2k}
\exp\bigg\{\alpha\Big(s-\frac{1}{2}-
\frac{x}{2\alpha}\Big)^2\bigg\}\>ds\\
\vbox{\kern.8cm}
&=\exp\bigg\{
-\frac{x}{2}-\frac{x^2}{4\alpha}
\bigg\}
\inte_{-\infty}^{+\infty}
\Big(\frac{x}{2\alpha}+it\Big)^{2k}
e^{-\alpha t^2}\,i\,dt\\
\vbox{\kern.8cm}
&=\exp\bigg\{
-\frac{x}{2}-\frac{x^2}{4\alpha}
\bigg\}
\frac{i}{(2\alpha)^k}
\inte_{-\infty}^{+\infty}
\Big(\frac{x}{\sqrt{2\alpha}}+it\Big)^{2k}
e^{-\frac{1}{2}t^2}\>dt\\
\vbox{\kern.8cm}
&=\exp\bigg\{
-\frac{x}{2}-\frac{x^2}{4\alpha}
\bigg\}
i\sqrt{\frac{\pi}{\alpha}}
\Big(\frac{1}{4\alpha}\Big)^k
H_{2k}\Big(\frac{x}{2\sqrt{\alpha}}\Big)
\end{split}
\end{equation*}
where for the last step in the above
calculation, we have made use of
the following formula
\begin{equation*}
\begin{split}
\inte_{-\infty}^{+\infty}
(\sqrt{2}\kern.03cm x+it)^{2k}
e^{-\frac{1}{2}t^2}\>dt=
\frac{\sqrt{2\pi}}{2^k}\kern.04cm
H_{2k}(x)
\end{split}
\end{equation*}
(see formula 8.951 in \cite{Gradshteyn}).
\end{proof}

In the next lemma, we will let
$g(s)$ be the logarithmic derivative
of $\chi(s)$ where 
\begin{equation}\label{ji}
\chi(s)=
2^s\pi^{s-1}
\sin\Big(\frac{\pi}{2}s\Big)
\Gamma(1-s)
\kern.8cm\hbox{and}\kern.8cm
\zeta(s)=
\chi(s)\zeta(1-s)
\end{equation}
is the functional equation satisfied
by the Riemann zeta function (see \cite{Tenenbaum},
pages 234 and 270). Furthermore, the
reader will notice that the proof of
the lemma is an application of
the saddle point method of complex
analysis.

\begin{lemma}\label{lB}
Let
\begin{eqnarray}\label{ge}
g(s)=\log 2\pi+\frac{\pi}{2}
\tan\Big(\frac{\pi s}{2}\Big)-
\psi(s)
\end{eqnarray}
where $\psi(s)=\Gamma'(s)/\Gamma(s)$ is
the digamma function.
Then, as $\xi=\sqrt{k/\alpha}\to\infty$,
we have that
\[
J=
\inte_{-\infty}^{+\infty}
\widehat{w}\Big(\frac{1}{2}-it\Big)
g\Big(\frac{1}{2}-it\Big)\>dt
=\log\frac{\xi}{2\pi}
+O\Big(\frac{\alpha^{\frac{3}{4}}
\log\xi}{\sqrt{k}}+
\frac{\log\xi}{e^{\sqrt{\alpha}}}
\Big).
\]
\end{lemma}

\begin{proof}
If we write
\[
J_1^\pm =
\inte_0^\infty
\widehat{w}\Big(\frac{1}{2}\pm it\Big)
\tan\Big(\frac{\pi}{4}\pm i\frac{t}{2}\Big)\>dt
\kern.5cm\hbox{and}\kern.5cm
J_2^\pm =
\inte_0^\infty
\widehat{w}\Big(\frac{1}{2}\pm it\Big)
\ss \psi\Big(\frac{1}{2}\pm it\Big)\>dt
\]
then we have
\[
J=
\log 2\pi+
\frac{\pi}{2}(J_1^-+J_1^+)+
(J_2^-+J_2^+).
\]
On the other hand,
\[
\tan\Big(\frac{\pi}{4}\pm
i\frac{t}{2}\Big)=
\pm i+O(e^{-t})
\kern.5cm\hbox{and}\kern.5cm
\psi\Big(\frac{1}{2}\pm it\Big)
=\log t\pm\frac{\pi}{2}i+O\Big(
\frac{1}{t}\Big)
\]
where the second of these formulas
follows from the asymptotic expansion for the
digamma function
(Abranowitz formula 6.3.18)
\[
\psi(z)=
\log z-\frac{1}{2z}-
\frac{1}{12z^2}+
\frac{1}{120z^4}-
\frac{1}{252z^6}+\cdots,
\]
for $z\to\infty$ in $|\arg z|<\pi$.

For the estimation of $J_2^\pm$
we first assume that
$|t-\xi|\leq\alpha^{-\frac{1}{4}}$.
Then we have
\begin{equation*}
\begin{split}
\widehat{w}\Big(
\frac{1}{2}+it\Big) &=
\frac{\alpha^{k+\frac{1}{2}}}
{\Gamma(k+\frac{1}{2})}
\exp\{-\alpha t^2+2k\log t\}\\
\vbox{\kern.8cm}
&=\frac{\alpha^{k+\frac{1}{2}}}
{\Gamma(k+\frac{1}{2})}
e^{-k}\Big(\frac{k}{\alpha}\Big)^k
\exp\Big\{-2\alpha(t-\xi)^2+
O\Big(\sqrt{\frac{\alpha^3}{k}}
(t-\xi)^3\Big)\Big\}\\
\vbox{\kern.8cm}
&=
\sqrt{\frac{\alpha}{2\pi}}
\exp\{-2\alpha(t-\xi)^2\}
\Big(1+O\Big(\frac{1}{k}\Big)\Big)
\Big(1+O\Big(\sqrt{\frac{\alpha^3}{k}}
(t-\xi)^3\Big)\Big).
\end{split}
\end{equation*}
Hence
\begin{eqnarray}\label{gorro}
\widehat{w}\Big(
\frac{1}{2}+it\Big) =
\sqrt{\frac{\alpha}{2\pi}}
\exp\{-2\alpha(t-\xi)^2\}
\Big(1+O\Big(\frac{\alpha^{\frac{3}{4}}}
{\sqrt{k}}\Big)\Big).
\end{eqnarray}
Furthermore,
\begin{eqnarray*}
&&
\inte_{\xi-\alpha^{-\frac{1}{4}}}^
{\xi+\alpha^{-\frac{1}{4}}}
\widehat{w}\Big(\frac{1}{2}
+it\Big)\log(t)\>dt\\
\vbox{\kern1cm}
&=&
\bigg\{1+O\Big(\frac{\alpha^{\frac{3}{4}}}
{\sqrt{k}}\Big)\bigg\}
\sqrt{\frac{\alpha}{2\pi}}
\inte_{\xi-\alpha^{-\frac{1}{4}}}^
{\xi+\alpha^{-\frac{1}{4}}}
\log(t)\ss e^{-2\alpha(t-\xi)^2}\>dt
\\
\vbox{\kern1cm}
&=&
\bigg\{1+O\Big(\frac{\alpha^{\frac{3}{4}}}
{\sqrt{k}}\Big)\bigg\}
\frac{1}{\sqrt{2\pi}}
\inte_{-\alpha^{\frac{1}{4}}}^
{\alpha^{\frac{1}{4}}}
\log\Big(\xi +\frac{t}{\sqrt{\alpha}}\Big)
\ss e^{-2 t^2}\>dt
\\
\vbox{\kern.8cm}
&=&
\bigg\{1+O\Big(\frac{\alpha^{\frac{3}{4}}}
{\sqrt{k}}\Big)\bigg\}
\bigg\{\log\xi+O\Big(\frac{\alpha^{-\frac{1}{4}}}{
\xi}\Big)\bigg\}
\bigg\{\frac{1}{2}+O(e^{-2\sqrt{\alpha}})\bigg\}\\
\vbox{\kern.8cm}
&=&
\frac{1}{2}\log\xi+
O\Big(\frac{\alpha^{\frac{3}{4}}
\log\xi}
{\sqrt{k}}\Big).
\end{eqnarray*}
For the estimation of the
integral defining $J_2^\pm$ over
the range of $t$ such that
$|t-\xi|>\alpha^{-\frac{1}{4}}$, we
first notice that
\begin{equation*}
\begin{split}
-\alpha t^2
+2k\log t+k-k\log\frac{k}{\alpha} &=
-\alpha(t-\xi)^2-2k-2k
\inte_\xi^t
\frac{t-y}{y^2}\>dy\\
\vbox{\kern0cm}
&\leq-\alpha(t-\xi)^2.
\end{split}
\end{equation*}
Hence we have
\begin{gather*}
\Big|\widehat{w}\Big(
\frac{1}{2}+it\Big)\Big| =
\frac{\alpha^{k+\frac{1}{2}}}
{\Gamma(k+\frac{1}{2})}
\exp\{-\alpha t^2+2k\log t\}\\
\vbox{\kern.7cm}
\ll\sqrt{\alpha}
\Big(\frac{\alpha}{k}\Big)^k
e^k\exp\{-\alpha t^2+2k\log t\}
\ll\sqrt{\alpha}
\exp\{-\alpha(t-\xi)^2\}.
\end{gather*}
Thus,
\begin{eqnarray*}
\sqrt{\alpha}
\inte_{\xi+\alpha^{-\frac{1}{4}}}^
\infty
\log(t)\ss e^{-\alpha(t-\xi)^2}\>dt 
\ll 
\inte_{\alpha^{\frac{1}{4}}}^\infty
t\ss e^{-t^2}\>dt=
\frac{1}{2}e^{-\sqrt{\alpha}}.
\end{eqnarray*}
On the other hand
\begin{gather*}
\sqrt{\alpha}
\inte_1^{\xi-\alpha^{-\frac{1}{4}}}
\log(t)\ss e^{-\alpha(t-\xi)^2}\>dt \leq
\sqrt{\alpha}\log(\xi)
\inte_{-\infty}^{-\alpha^{-\frac{1}{4}}}
e^{-\alpha t^2}\>dt
\ll
\log(\xi)
e^{-\sqrt{\alpha}}.
\end{gather*}
Therefore
\[
J_2^-+J_2^+=
\log\xi+
O\Big(\frac{\alpha^{\frac{3}{4}}
\log\xi}{\sqrt{k}}+
\frac{\log\xi}{e^{\sqrt{\alpha}}}
\Big).
\]
In a similar fashion we see that
$J_1^-+J_1^+=
O(e^{-\xi})$.
This finishes the proof of the lemma.
\end{proof}

\begin{lemma}\label{lC}
We have that
\[
c_1\sum_{n=1}^\infty
\gamma^{2k}_n\ss e^{-\alpha \gamma^2_n}=
\sqrt{\frac{\alpha}{2\pi}}
\sum_{|\gamma_n-\xi|\leq\alpha^{-\frac{1}{4}}}
\exp\{-2\alpha(\gamma_n-\xi)^2\}+
O\Big(\frac{\alpha}
{\sqrt{k}}\log\xi+
e^{-\sqrt{\alpha}}
\Big).
\]
\end{lemma}

\begin{proof}
We write $s_1+s_2$ for the sum we want
to estimate, where $s_1$ is the sum
over $\gamma_n$ such that $|\gamma_n-\xi|\leq
\alpha^{-\frac{1}{4}}$ and $s_2$ is
sum over $\gamma_n$ such that $|\gamma_n-\xi|>
\alpha^{-\frac{1}{4}}$. Because
of equation (\ref{gorro}), for $s_1$
we have,
\begin{eqnarray*}
s_1 &=&
\sqrt{\frac{\alpha}{2\pi}}
\Big\{1+O\Big(\frac{\alpha^{\frac{3}{4}}}{\sqrt{k}}
\Big)\Big\}
\sum_{|\gamma_n-\xi|\leq\alpha^{-\frac{1}{4}}}
\exp\{-2\alpha(\gamma_n-\xi)^2\}\\
\vbox{\kern.8cm}
&=&
\sqrt{\frac{\alpha}{2\pi}}
\sum_{|\gamma_n-\xi|\leq\alpha^{-\frac{1}{4}}}
\exp\{-2\alpha(\gamma_n-\xi)^2\}+O\Big(
\frac{\alpha^{\frac{3}{4}}}{\sqrt{k}}
\sqrt{\alpha}
\sum_{|\gamma_n-\xi|\leq\alpha^{-\frac{1}{4}}}1\Big).
\end{eqnarray*}
We recall that
if $N(t)$ is the number of zeros of
$\zeta(s)$ such that $0<\gamma_n\leq t$,
then $N(t)\asymp t\log t$. Therefore,
\begin{eqnarray*}
s_1 &=&
\sqrt{\frac{\alpha}{2\pi}}
\sum_{|\gamma_n-\xi|\leq\alpha^{-\frac{1}{4}}}
\exp\{-2\alpha(\gamma_n-\xi)^2\}+O\Big(
\frac{\alpha}{\sqrt{k}}\log\xi
\Big).
\end{eqnarray*}
By proceeding as in the proof of
lemma~\ref{lB} we see that
$s_2\ll\exp\{-\sqrt{\alpha}\}$.
This finishes the proof of lemma~\ref{lC}.
\end{proof}

With the help of the previous lemmas,
we can proceed to the proof of theorem 1.
From Perron inversion formula 
(\ref{mellin}) we have
\[
m(1)=
-\frac{1}{2\pi i}
\inte_{2-i\infty}^{2+i\infty}
\widehat{w}(s) 
\frac{\zeta'(s)}{\zeta(s)}\>ds.
\]
By shifting the vertical line of
integration from $\sigma=2$ to 
$\sigma=-1$, we get
\begin{equation*}
\begin{split}
m(1)=c_1\ss\frac{(-1)^{k+1}}{4^k}
e^{\frac{1}{4}\alpha}-
2c_1\sum_{n=1}^\infty
\gamma^{2k}_n\ss e^{-\alpha \gamma^2_n}
-\frac{1}{2\pi i}
\inte_{-1-i\infty}^{-1+i\infty}
\widehat{w}(s) 
\frac{\zeta'(s)}{\zeta(s)}\>ds.
\end{split}
\end{equation*}
Because of lemma~\ref{lC} 
the second term on the
right hand side of this equations
can be written as $2S(\xi,k)+E$.
Because of the functional equation
(\ref{ji}),
we can write the last integral as
\[
-\frac{1}{2\pi i}
\inte_{2-i\infty}^{2+i\infty}
\widehat{w}(1-s)\Big( 
g(s)-
\frac{\zeta'(s)}{\zeta(s)}\Big)
\>ds
\]
where $g(s)$ is as in equation
(\ref{ge}).
If we set $\widehat{k}(s)=\widehat{w}(1-s)$,
then, from the operational
properties of the Mellin transform, we have that 
$k(x)=x^{-1}w(x^{-1})$ and therefore,
\begin{equation*}
\begin{split}
\frac{1}{2\pi i}
\inte_{2-i\infty}^{2+i\infty}
\widehat{w}(1-s)
\frac{\zeta'(s)}{\zeta(s)}
\>ds &=
-\sum_{n=1}^\infty\Lambda(n)
\frac{1}{2\pi i}
\inte_{2-i\infty}^{2+i\infty}
\widehat{k}(s)\ss n^{-s}\>ds\\
\vbox{\kern.8cm}
&=
-\sum_{n=1}^\infty\frac{\Lambda(n)}{n}
w\Big(\frac{1}{n}\Big)
\\
\vbox{\kern.8cm}
&=
-c_2\sum_{n=1}^\infty
\frac{\Lambda(n)}{\sqrt{n}}
\exp\bigg\{
\frac{-1}{4\alpha}\log^2n\bigg\}
H_{2k}\Big(\frac{\log n}{2
\sqrt{\alpha}}\Big)\\
\vbox{\kern.7cm}
&=-m(1).
\end{split}
\end{equation*}
Hence,
\[
2\ss m(1)=c_1\ss\frac{(-1)^{k+1}}{4^k}
e^{\frac{1}{4}\alpha}-
2\ss S(\xi,k)+
\frac{1}{2\pi}\tilde{J}+
O\Big(\frac{\alpha
\log\xi}{\sqrt{k}}+
\frac{\log\xi}{e^{\sqrt{\alpha}}}
\Big)
\]
where 
\begin{eqnarray*}
\frac{1}{2\pi}\tilde{J}
&=& \frac{-1}{2\pi i}
\inte_{2-i\infty}^{2+i\infty}
\widehat{w}(1-s)\ss
g(s)
\>ds=
\frac{-1}{2\pi}
\inte_{-\infty}^{+\infty}
\widehat{w}\Big(\frac{1}{2}-it\Big)
g\Big(\frac{1}{2}-it\Big)\>dt\\
\vbox{\kern.7cm}
&=&\frac{1}{2\pi}
\log\frac{\xi}{2\pi}+
O\Big(\frac{\alpha
\log\xi}{\sqrt{k}}+
\frac{\log\xi}{e^{\sqrt{\alpha}}}
\Big)
\end{eqnarray*}
the last equality being a consequence
of lemma~\ref{lB}.
This finishes the proof.


\begin{thebibliography}{99}
\bibitem{Abramowitz}
Abramowitz, M.; Stegun, I.A. 
\emph{Handbook of mathematical functions with 
formulas, graphs, and mathematical tables}. 
National Bureau of Standards Applied 
Mathematics Series , No. 55, 1964.
\bibitem{Balanzario}
Balanzario, E.P.; Cárdenas Romero, E.D.;
Chacón Serna, R.
\emph{A smooth version of Landau’s 
explicit formula}.
Manuscript submitted, 2023.
\bibitem{Gradshteyn}
Gradshteyn, I.S.; Ryzhik, I.M. 
\emph{Table of integrals, series, and products}. 
Academic Press, 2007.
\bibitem{Ivic}
Ivić, A.
\emph{The Riemann zeta-function}.
Dover Publications, Inc., 2003.
\bibitem{Landau}
Landau, E. 
\emph{Über die Nullstellen der Zetafunktion}.
Math. Ann. 71 (1912)
548-564.
\bibitem{Tenenbaum}
Tenenbaum, G.
\emph{Introduction to analytic and
probabilistic number theory}.
American Mathematical Society, 2015.
\end{thebibliography}
\end{document}